\documentclass[11pt,letterpaper]{article}
\usepackage{amsfonts}
\usepackage[left=1in,top=1in,right=1in,bottom=1in]{geometry}
\usepackage{setspace}
\usepackage{bbm}
\usepackage{amsmath,amsthm,amssymb}
\usepackage{graphicx}
\usepackage{color}
\usepackage{soul}

\newcommand{\R}{\mathbb{R}}

\newcommand{\dif}{\textrm{d}}

\newcommand{\I}[1]{1_{#1}}

\renewcommand{\Vec}[1]{\textrm{Vec}\left(#1\right)}

\newcommand{\tr}[1]{\textrm{tr}\left(#1\right)}

\newcommand{\rank}[1]{\textnormal{rank}\left(#1\right)}
\newcommand{\norm}[2]{\left|\left|#2\right|\right|_{#1}}
\newtheorem{definition}{Definition}
\newtheorem{lemma}{Lemma}
\newtheorem{proposition}{Proposition}
\newtheorem{theorem}{Theorem}
\newtheorem{remark}{Remark}

\begin{document}

\author{Mario Diaz\thanks{%
Department of Mathematics and Statistics, Queen's University,\ Kingston, ON
Canada, \textit{13madt@queensu.ca}} \and Victor P\'{e}rez-Abreu \thanks{%
Department of Probability and Statistics, CIMAT, Guanajuato, Mexico, \textit{%
pabreu@cimat.mx}}}
\title{\textbf{Random Matrix Systems with Block-Based Behavior and
Operator-Valued Models}}
\date{\today}
\maketitle

\begin{abstract}
A model to estimate the asymptotic isotropic mutual information of a
multiantenna channel is considered. Using a block-based dynamics and the
angle diversity of the system, we derived what may be thought of as the
operator-valued version of the Kronecker correlation model. This model turns
out to be more flexible than the classical version, as it incorporates both
an arbitrary channel correlation and the correlation produced by the
asymptotic antenna patterns. A method to calculate the asymptotic isotropic
mutual information of the system is established using operator-valued free
probability tools. A particular case is considered in which we start with
explicit Cauchy transforms and all the computations are done with diagonal
matrices, which make the implementation simpler and more efficient.
\end{abstract}



\section{Introduction}

Random matrices and free probability are areas of applied probability with
increasing importance in the area of multiantenna wireless systems, see for
example \cite{Coulliet2011}. One key problem in the stochastic analysis of
these systems has been the study of their asymptotic performance with
respect to the number of antennas. The first answer to this question is the
groundbreaking work by Telatar \cite{Telatar1999}, who, describing the
system as a random matrix with statistically independent entries, showed
that the capacity of this system is infinite. Since this independence
condition might be restrictive, several further proposals have been made
over the past decade. As a result, a few models have emerged to take into
account some instances of correlation in the system \cite{Foschini2000}, 
\cite{Mestre2003}, \cite{Tulino2005}.

Operator-valued free probability theory has proved to be a powerful tool to
study block random matrices \cite{Benaych2009,Shlyakhtenko1996}. This has
made possible to analyze certain systems exhibiting some block-based
dynamics \cite{Far2008,Speicher2012}. With recent developments in
operator-valued free probability theory \cite{Belinschi2013a,Belinschi2013b}%
, simple matricial iterative algorithms now allow us to find the asymptotic
spectrum of sums and products of free operator-valued random variables.

The purpose of the present paper is show the significance of these new tools
by studying a particular application in wireless communications. In
particular, we study an operator-valued Kronecker correlation model based on
an arbitrarily correlated finite dimensional multiantenna channel. From a
block matrix dynamics and a parameter related to the angle diversity of the
system, an operator-valued equivalent is derived and then a method to
calculate the asymptotic isotropic mutual information is developed using
tools from operator-valued free probability. The model allows using
information related to the asymptotic antenna patterns of the system. To our
best knowledge, this the first time that a model with these characteristics
is analyzed.

More precisely, a multiantenna system is an electronic communication setup
in which both the transmitter and the receiver use several antennas. The
input and the output of the system can be thought of as complex vectors $%
u=(u_{1},\cdots,u_{n_{T}})^{\top}$ and $v=(v_{1},\cdots,v_{n_{R}})^{\top} $,
where $n_{T}$ is the number of transmitting antennas and $n_{R}$ is the
number of receiving antennas. The system response is characterized by the
linear model 
\begin{equation*}
v=Hu+w, 
\end{equation*}
where $H$ is an $n_{R}\times n_{T}$ complex random matrix that models the
propagation coefficients from the transmitting to the receiving antennas and 
$w$ is a circularly symmetric Gaussian random vector with independent
identically distributed unital power entries.

In a correlated multiantenna system, there is correlation between the
propagation coefficients. Namely, the random matrix $H$ is such that the
random variables $\{H_{i,j}:i=1,\ldots,n_{R};j=1,\ldots,n_{T}\}$ are not
necessarily independent. It is customary to take the random variables
composing $H$ with circularly symmetric Gaussian random law \cite%
{Telatar1999}. In this context, the joint distribution of the entries of $H$
depends only on the covariance function $\sigma(i,j;i^{\prime},j^{\prime}):=%
\mathbb{E}\left( H_{i,j}\overline{H}_{i^{\prime},j^{\prime}}\right) $ for $%
i,i^{\prime}\in\{1,\ldots,n_{R}\}$ and $j,j^{\prime}\in\{1,\ldots,n_{T}\}$.

For a fixed rate $n_{T}/n_{R}$, it is known that the capacity of a
multiantenna system grows linearly with the number of antennas of the system
as long as the matrix $H$ has independent entries \cite{Telatar1999}. This
shows the well-behaved scalability properties of multiantenna systems.
However, correlation may have a negative effect on the performance of the
system. Therefore, it is necessary to estimate quantitatively the effect
that correlation may have.

Throughout this paper we will assume that the transmitter uses an
isotropical scheme, i.e., $\mathbb{E}\left( uu^{\ast}\right) =\dfrac{P}{n_{T}%
}\mathrm{I}_{n_{T}}$ where $P$ is the transmitter power. In this case, a
canonical way to quantify the effect of correlation is by means of the
asymptotic isotropic mutual information\footnote{%
Observe that this quantity is not the capacity of the system since the input
is restricted to be isotropic.} per antenna \cite{Telatar1999}.
Specifically, suppose that $H_{1}:=H$ and for each $N\geq2$ the $%
n_{R}^{(N)}\times n_{T}^{(N)}$ random matrix $H_{N}$ describes the channel
behavior when there are $n_{T}^{(N)}$ transmitting antennas and $n_{R}^{(N)}$
receiving antennas. Moreover, suppose that both $(n_{T}^{(N)})_{N\geq1}$ and 
$(n_{R}^{(N)})_{N\geq1}$ are increasing sequences and $%
n_{T}^{(N)}/n_{R}^{(n)}$ converges to a positive real number. Then, the
asymptotic isotropic mutual information per antenna $I_{\infty}$ is 
\begin{equation*}
I_{\infty}=\lim_{N\rightarrow\infty}\mathbb{E}\left( \frac{1}{n_{R}^{(N)}}%
\log\det\left( I+\frac{P}{n_{T}^{(N)}}H_{N}H_{N}^{\ast}\right) \right) , 
\end{equation*}
as long as the limit exists. A common phenomena in random matrix theory is
that the sequence of arguments in the expected value above converges almost
surely to a constant, and under mild conditions also in mean. Therefore, the
asymptotic isotropic mutual information per antenna is given, essentially,
by the a.s. limit of the aforementioned sequence.

Therefore, in order to find $I_{\infty}$, it is necessary to derive a model
for the sequence of random matrices $(H_{N})_{N\geq1}$ that approximates the
channel behavior in the finite size regime and then compute the asymptotic
quantity $I_{\infty}$.

In this paper we use an alternative method described in four steps:

\begin{enumerate}
\item Assign an operator-valued matrix $\mathbf{H}$ to the matrix $H$;

\item Compute the operator-valued Cauchy transform of $\mathbf{HH}^{\ast}$;

\item Via the Stieltjes inversion formula, recover the distribution of $%
\mathbf{HH}^{\ast}$, call it $F$;

\item Compute $I_{\infty}$ as 
\begin{equation}
I_{\infty}=\int\log(1+P\xi)F(\text{d}\xi).   \label{AsymptoticCapacity}
\end{equation}
\end{enumerate}

The operator-valued matrix $\mathbf{H}$ can be thought of as the asymptotic
operator-valued equivalent of the channel $H$ \cite{Speicher2012}. In this
sense, the common approach consists of giving a model for the finite size
regime, computing the mutual information, and taking the limit. On the other
hand, the alternative approach takes \textit{limits in the model}, replacing
matrices by operator-valued matrices, and then calculates the mutual
information. Of course, these approaches are intimately related. Actually,
in the traditional case, they provide the same results\footnote{%
For example, in the iid case, we know that the empirical spectral
distribution of $H_{N}H_{N}^{\ast}/n_{T}$ converges in distribution almost
surely to the Marchenko--Pastur distribution \cite{Telatar1999}. This is
equivalent to saying that $H_{N}H_{N}^{\ast}/n_{T}$ converges in
distribution to a noncommutative random variable whose analytical
distribution $F$ is the corresponding Marchenko--Pastur distribution, which
gives the asymptotic mutual information (1).}, but we prefer the latter
approach since it is conceptually easier to understand and carry out,
providing a powerful tool for modelling.

We will see that this way of thinking goes well with channels exhibiting a
block-based behavior. In particular, the operator-valued matrix assigned in
step 1 carries the block structure of the channel and some other features of
the system. In the example analyzed here, these features include the effect
of the asymptotic antenna patterns and the inclusion of the starting finite
dimensional channel correlation. To illustrate the kind of tools that may be
useful in the assigning process at step 1, in the next section we retrieve a
block-based Kronecker model from an angular-based model and derive the
operator-valued equivalent $\mathbf{H}$.

In Section \ref{Section:Model} we derive the proposed operator-valued
Kronecker correlation model. In Section \ref{Section:AsymptoticCapacity} we
discuss the asymptotic isotropic mutual information of our model using tools
from operator-valued free probability. In Section \ref{Section:DiagonalCase}
we consider a particular example where the implementation is simple but at
the same time flexible enough to be applied in several interesting cases,
like some symmetric channels. In Section \ref{Section:NumericalComparison}
we compare, through the example of a finite dimensional system, the mutual
information predicted by the usual Kronecker correlation model against the
results from the proposed operator-valued alternative. In Appendix A we
summarize the notation, the background, and the prerequisites from
operator-valued free probability theory. In Appendix B we prove Theorem \ref%
{Thm:ExtremalCases} on two extreme behaviors of the model regime. In
Appendix C we compute some of the operator-valued Cauchy transforms required
in this paper.

\section{The Angular Based Model and Its Operator-Valued Equivalent}

\label{Section:Model}

The proposed model to approximate the channel behavior in the finite size
regime is derived as follows. Suppose that for a fixed $N\in\mathbb{N}$,
each antenna of the original system is replaced by $N$ new antennas located
around the position of the original one. Thus, the new system has $n_{T}N$
transmitting and $n_{R}N$ receiving antennas. Figure \ref{fig:Antennas}
shows the original system for $n_{T}=1$ and $n_{R}=2$ together with the
corresponding virtual one for $N=2$.

\begin{figure}[th]
\centering
\includegraphics[width=0.7\textwidth]{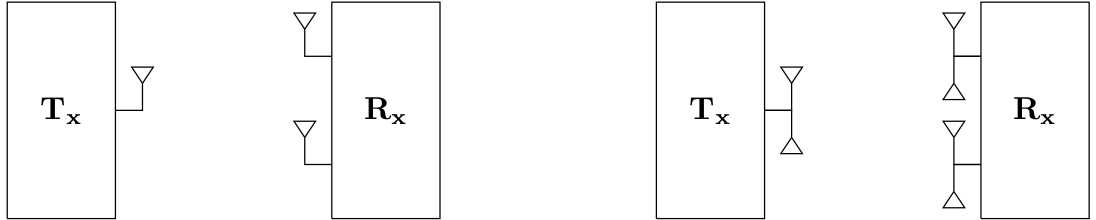}
\caption{On the left the original $1\times2$ system. On the right the
virtual $2\times4$ system corresponding to $N=2$.}
\label{fig:Antennas}
\end{figure}

For any given $N\in\mathbb{N}$, the channel matrix $H_{N}$ for this $%
n_{R}N\times n_{T}N$ system will have the form 
\begin{equation*}
H_{N}=\left( 
\begin{matrix}
H_{N}^{(1,1)} & \cdots & H_{N}^{(1,n_{T})} \\ 
\vdots & \ddots & \vdots \\ 
H_{N}^{(n_{R},1)} & \cdots & H_{N}^{(n_{R},n_{T})}%
\end{matrix}
\right) , 
\end{equation*}
where $H_{N}^{(i,j)}$ is the $N\times N$ matrix whose entries are the
coefficients between the new antennas that come from the original $i$
receiving and $j$ transmitting antennas.

\subsection{Statistics of the channel and block matrix structure}

We now derive a model for $H_{N}=(H_{N}^{(i,j)})_{i,j}$ that takes into
account the statistics of the channel matrix $H$ and the block structure
exhibited above. First, fix a block $H_{N}^{(i,j)}$, and for notational
simplicity denote it by $A$. This matrix $A$ should reflect the behavior of
a scalar channel between two antennas of the original system when these are
replaced by $N$ antennas each.

In a regime of a very high density of antennas per unit of space, any two
pairs of antennas close enough are likely to experience very similar fading.
Since as $N\rightarrow\infty$ the new antennas are closer to each other,
then the propagation coefficients between them are prone to be correlated.
As an extreme case, we suppose that all the propagation coefficients between
the antennas involved in $A$ have the same norm, and without loss of
generality we set this to be one\footnote{%
Latter, we will incorporate the effect of these norms in the covariance of
our operator-valued equivalent.}. This means that these coefficients produce
the same power losses and the differences between them come from the
variation that they induce in the signal's phases. With this in mind, we
will suppose that for $1\leq k,l\leq N$, 
\begin{equation*}
A_{k,l}=\exp(\gamma\mathrm{i}\theta_{k,l}) 
\end{equation*}
where $\mathrm{i}=\sqrt{-1}$, $\theta_{k,l}$ is a real random variable and $%
\gamma>0$ is a physical parameter that reflects the statistical variation of
the phases of the incoming signals. In some geometrical models, this
statistical variation of the phases has been used, along with the angle of
arrival and the angle spread, to study the capacity of multiantenna channels 
\cite{Foschini2000}.

Some of the physical factors that have the most impact on the correlation of
an antenna array are related to either the physical parameters of the
antennas or to local scatterers. Since these factors are different for each
end of the communication link then, borrowing the intuition from the usual
Kronecker model, it is natural to take the matrix $\theta=(\theta
_{k,l})_{k,l=1}^{N}$ as a separable or Kronecker correlated matrix, that is, 
\begin{equation*}
\theta=RXT 
\end{equation*}
where $R$ and $T$ are the square roots of suitable correlation matrices and $%
X$ is a random matrix with independent entries having the standard Gaussian
distribution. It is important to point out that $A$ is not Kronecker
correlated.

\subsection{Extreme regimes of the parameter $\protect\gamma$ of the system}

From a modelling point of view, the case $\gamma\rightarrow\infty$
represents the situation in which the environment is rich enough to ensure a
high diversity in the angles of the propagation coefficients. On the other
hand, the case $\gamma\rightarrow0$ represents a system in which the
propagation coefficients in the given block are almost the same.
Intuitively, the first case is better in terms of $\gamma$, since we should
be able to recover the multiantenna diversity via the angle diversity; while
in the second case we almost lose the diversity advantage of a multiantenna
system over a single antenna system.

In these limiting cases the following holds. We denote by $\lambda
_{1}(\cdot )\geq \cdots \geq \lambda _{N}(\cdot )$ the ordered eigenvalues
of an Hermitian matrix.

\begin{theorem}
\label{Thm:ExtremalCases} Assume that $R$ and $T$ are full rank. For $N$
fixed, as $\gamma \rightarrow \infty $, 
\begin{equation*}
(\lambda _{1}(AA^{\ast }),\ldots ,\lambda _{N}(AA^{\ast }))\Rightarrow
(\lambda _{1}(UU^{\ast }),\ldots ,\lambda _{N}(UU^{\ast }))
\end{equation*}%
where $U$ is a matrix with i.i.d. entries with uniform distribution on the
unit circle.

Suppose that $(\gamma _{N})_{N\geq 1}$ is a sequence of positive real
numbers such that $\gamma _{N}\rightarrow 0$ as $N\rightarrow \infty $.
Then, almost surely, $F^{\gamma _{N}^{-2}AA^{\ast }}\Rightarrow F$ as $%
N\rightarrow \infty $ where $F$ is the asymptotic eigenvalue distribution of 
$\theta \theta ^{\ast }$.
\end{theorem}

\begin{proof}
See Appendix B.
\end{proof}

Observe that in the second part of the previous theorem both $A$ and $\theta$
depend on $N$ as they are $N\times N$ matrices.

This means that the entries of the matrix $A$ become uncorrelated as $%
\gamma\rightarrow\infty$, and, by universality, the spectrum of $A$ must
behave similar to the spectrum of a standard Gaussian matrix of the same
size. Observe that in this limiting case, we arrive at the well known case
of i.i.d. entries, i.e., the canonical model of a multiantenna system \cite%
{Telatar1999}. As was mentioned before, in this situation the environment
has a high diversity in the angles of the propagation coefficients, and thus
it is natural that the system behaves as in the i.i.d. case.

On the other hand, when $\gamma \rightarrow 0$, the bulk of $AA^{\ast }$ is
close to that of $\gamma ^{2}\theta \theta ^{\ast }=\gamma ^{2}RXT^{2}XR$.
This suggests approximating $A\approx \gamma RXT$. Note that this limiting
case leads to the well known Kronecker correlation model \cite{Tulino2005}.
In the spirit of a worst case analysis, we will use $A=\gamma RXT$ in what
follows.

\begin{remark}
In the proof of Theorem \ref{Thm:ExtremalCases} we only used the fact that $%
\theta\theta^*$ has an asymptotic eigenvalue distribution with compact
support and $\left\vert \left\vert \theta\theta^*\right\vert \right\vert_{op}
$ converges a.s. as $N\to\infty$. Therefore, under these mild conditions,
the same analysis yields to the approximation $A=\gamma\theta$ for any model 
$\theta$.
\end{remark}

\subsection{Operator-valued free probability modelling}

In terms of the asymptotic behavior of the spectrum and invoking ideas from
random matrix theory and free probability, let $(\mathcal{C},\varphi)$ be a
noncommutative probability space where the algebra $\mathcal{C}$ has unit $%
\mathbf{1}_{\mathcal{C}}$ (see Appendix A). We can model the matrix $%
A=\gamma RXT$ by means of a noncommutative random variable $a$ in $\mathcal{C%
}$ such that $a=rxt$ where $r$ and $t$ are in $\mathcal{C}$ such that $r^{2}$
and $t^{2}$ are the limits in distribution of $R^{2}$ and $T^{2}$
respectively, and $x$ is a circular operator with a given variance. Since
the matrices $R^{2}$ and $T^{2}$ depend on separate sides of the
communication link, and in some contexts, such as mobile communications, the
transmitter and receiver are not in any particular orientation with respect
to each other, we can assume that the eigenmodes of this matrices are in
standard position. In particular, this means that the distributional
properties of $R$ and $T$ are invariant under random rotations, i.e., $(R,T)%
\overset{\mathrm{d}}{=}(R,UTU^{\ast})$ where $U$ is a Haar distributed
random matrix independent from $R$ and $T$. The latter implies that $r$ and $%
t$ are free \cite{Coulliet2011}, and by Voiculescu's theorem \cite{Hiai2000}
they both are free from $x$.

If we use the noncommutative random variable representation, as we did with $%
A$, for every block $H_{N}^{(i,j)}$ with $i\in\{1,\ldots,n_{R}\}$ and $%
j\in\{1,\ldots,n_{T}\}$, then 
\begin{equation*}
\left(H_{N}^{(i,j)}:i=1,\ldots,n_{R};j=1,\ldots,n_{T}\right) \overset{%
\textnormal{dist}}{\longrightarrow} \left(
r_{i,j}x_{i,j}t_{i,j}:i=1,\ldots,n_{R};j=1,\ldots,n_{T}\right) ,
\end{equation*}
where $r_{i,j}$, $t_{i,j}$ and $x_{i,j}$ are the corresponding correlation
and circular random variables for the block $H_{N}^{(i,j)}$. By the same
argument as before, we assume that the families $\{r_{i,j}:i,j\}$ and $%
\{t_{i,j}:i,j\}$ are free. In this way, for any $m\in\mathbb{N}$ 
\begin{equation*}
\lim_{N\rightarrow\infty}\mathbb{E}\left( \text{tr}_{n_{R}N}\left(
(H_{N}H_{N}^{\ast})^{m}\right) \right) =(\text{tr}_{n_{R}}\circ E)\left( (%
\mathbf{HH}^{\ast})^{m}\right) 
\end{equation*}
where 
\begin{equation*}
\mathbf{H}=\left( 
\begin{matrix}
r_{1,1}x_{1,1}t_{1,1} & \cdots & r_{1,n_{T}}x_{1,n_{T}}t_{1,n_{T}} \\ 
\vdots & \ddots & \vdots \\ 
r_{n_{R},1}x_{n_{R},1}t_{n_{R},1} & \cdots & 
r_{n_{R},n_{T}}x_{n_{R},n_{T}}t_{n_{R},n_{T}}%
\end{matrix}
\right) . 
\end{equation*}

Let $\mathrm{I}_{n_{R}}\otimes\text{tr}_{N}:\textnormal{M}_{n_{R}N}\left( 
\mathbb{C}\right) \rightarrow\textnormal{M}_{n_{R}}\left( \mathbb{C}\right) $
be the linear map determined by 
\begin{equation*}
(\mathrm{I}_{n_{R}}\otimes\text{tr}_{N})(E_{i,j}\otimes A)=\text{tr}%
_{N}\left( A\right) E_{i,j}
\end{equation*}
where $E_{i,j}$ is the $i,j$-unit matrix in $\textnormal{M}_{n_{R}}\left( 
\mathbb{C}\right) $ and $A$ is any $N\times N$ matrix. It is clear that $%
\text{tr}_{n_{R}N}=\text{tr}_{n_{R}}\circ(\mathrm{I}_{n_{R}}\otimes \text{tr}%
_{N})$. For every $N\in\mathbb{N}$, $H_{N}H_{N}^{*}$ belongs to the $%
\textnormal{M}_{n_{R}}\left( \mathbb{C}\right) $-valued probability space%
\footnote{%
Here, $\textnormal{M}_{n_{R}N}\left( \mathbb{C}\right) $ is in fact an
algebra of $n_{R}N\times n_{R}N$ random matrices over the complex numbers.
This is the only time we use this abuse of notation.} $(\textnormal{M}%
_{n_{R}N}\left( \mathbb{C}\right) ,\mathbb{E}\circ(\mathrm{I}_{n_{R}}\otimes 
\text{tr}_{N}))$ and $\mathbf{HH}^{*}$ to the $\textnormal{M}_{n_{R}}\left( 
\mathbb{C}\right) $-valued probability space $(\textnormal{M}_{n_{R}}\left( 
\mathcal{C}\right) ,E) $ where $E:=\mathrm{I}_{n_{R}}\otimes\varphi$.

Moreover, we will restrict ourselves to working with asymptotic eigenvalue
distributions with compact support, which allows us to work within the
framework of a $C^{\ast}$-probability space. In this context, convergence in
distribution implies weak convergence of the corresponding analytic
distributions \cite{Nica2006}, see also Appendix A.

In the derivation of this model, we can observe that all the $%
r_{k,1},\ldots,r_{k,n_{T}}$ depend on the new antennas around the original $k
$th receiving antenna, and thus it is reasonable to take all them equal to
some random variable $r_{k}$. Proceeding with this reasoning at the
transmitter side, we conclude that 
\begin{equation}
\mathbf{H}=\mathbf{RXT}   \label{Hmodel}
\end{equation}
where $\mathbf{R}=\textnormal{diag}\left( r_{1},\ldots,r_{n_{R}}\right) $
and $\mathbf{T}=\textnormal{diag}\left( t_{1},\ldots,t_{n_{T}}\right) $ are
the operators associated to the correlation structure of the antennas at
each side, and $\mathbf{X}=(x_{i,j})_{i,j}$. Thus, we can think of this
model as the operator-valued version of the Kronecker correlation model for
multiantenna systems.

Moreover, let $\Sigma^{2}$ be the correlation matrix\footnote{%
With respect to $E:=\mathbf{1}\otimes\varphi$.} of $\Vec{\mathbf{X}}$, i.e., 
$E\left( \Vec{\mathbf{X}}\Vec{\mathbf{X}}^{\ast}\right) =\Sigma^{2}$. In
terms of the model, $\Sigma^{2}$ must reflect the correlation structure of
the channel matrix $H$ and the parameter $\gamma$ of the system. A
reasonable way to do this is by setting $\Sigma^{2}=\gamma^{2}\mathbb{E}%
\left( \Vec{H}\Vec {H}^{\ast}\right) $. In the regime $\gamma\rightarrow0$,
the latter implies that the mutual information decreases proportionally to $%
\gamma^{2} $. Since we can incorporate the constant $\gamma$ into the
correlation operator-valued matrices $\mathbf{R}$ and $\mathbf{T}$, for
notational simplicity we will set $\gamma=1$ in our discussion, thus we will
take $\Sigma^{2}=\mathbb{E}\left( \Vec{H}\Vec{H}^{\ast}\right) $.
Nonetheless, remember that the model derivation was made in the regime $%
\gamma\rightarrow0$.

Observe that each $r_{k}$ depends on the new antennas around the original $k$%
th receiving antenna. Thus the distribution of $r_{k}$ will depend strongly
on the specific geometric distribution of the new antennas. For example, if
all the new antennas are located in exactly the same place\footnote{%
Of course this is physically impossible.} as the original antenna, we would
obtain that the distribution of $r_{k}$ must be zero. In the case where the
antennas are collinear and equally spaced, we can use some class of Toeplitz
operators as shown in \cite{Mestre2003}. A similar argument can be used for
the transmission operators.

\begin{remark}
Observe that in this way we have incorporated the finite dimensional
statistics in our operator-valued equivalent. Moreover, the correlation
matrix $\Sigma^{2}$ does not need to be separable, i.e., with a Kronecker
structure. This shows that the operator-valued Kronecker model is slightly
more flexible than the classical version: it allows an arbitrary correlation
resulting from the channel, and it also allows different correlations for
different regions of the transmitter and receiver antenna arrays, which in
our notation is encoded in the matrices $\mathbf{T}$ and $\mathbf{R}$.
\end{remark}

\section{Asymptotic Isotropic Mutual Information Analysis}

\label{Section:AsymptoticCapacity}

In this section we derive a method to calculate the asymptotic isotropic
mutual information (\ref{AsymptoticCapacity}) of our model using the tools
of operator-valued free probability. For simplicity of exposition, in what
follows we will take $n_{R}=n_{T}=n$. Note that if, for example, $n_{R}<n_{T}
$, then we can proceed by just taking $n=n_{T}$ and by taking $r_{k}$ equal
to $0$ for $k>n_{R}$. Let $\mathbf{R}$, $\mathbf{X}$ and $\mathbf{T}$ as in (%
\ref{Hmodel}). The goal is to find the distribution $F$ of $\mathbf{HH}%
^{\ast}$ in (\ref{AsymptoticCapacity}) by means of the $\textnormal{M}%
_{n}\left( \mathbb{C}\right) $-valued Cauchy transform of $\mathbf{HH}^{\ast}
$ (see Appendix A).

Using the symmetrization technique \cite{Speicher2012}, we define $\widehat{%
\mathbf{H}}$ as 
\begin{equation*}
\widehat{\mathbf{H}}=\left( 
\begin{matrix}
0 & \mathbf{H} \\ 
\mathbf{H}^{\ast} & 0%
\end{matrix}
\right) . 
\end{equation*}
Notice that the distribution of $\widehat{\mathbf{H}}^{2}$ is the same as
the distribution of $\mathbf{HH}^{\ast}$, and that $\widehat{\mathbf{H}}$ is
selfadjoint. Since all the odd moments of $\widehat{\mathbf{H}}$ are 0, the
distribution of $\widehat{\mathbf{H}}$ is symmetric.\medskip

We can then obtain the $\textnormal{M}_{2n}\left( \mathbb{C}\right) $-valued
Cauchy transform (\ref{CTOper}) of $\widehat{\mathbf{H}}^{2}$ from the
corresponding transform of $\widehat{\mathbf{H}}$ using the formula \cite%
{Nica2006} 
\begin{equation*}
G_{\widehat{\mathbf{H}}}(\zeta\mathrm{I}_{2n})=\zeta G_{\widehat{\mathbf{H}}%
^{2}}(\zeta^{2}\mathrm{I}_{2n}), 
\end{equation*}
where $\zeta\in\mathbb{C}$ and $\mathrm{I}_{2n}$ is the identity matrix in $%
\textnormal{M}_{2n}\left( \mathbb{C}\right) $. Since 
\begin{equation*}
\widehat{\mathbf{H}}=\left( 
\begin{matrix}
\mathbf{R} & 0 \\ 
0 & \mathbf{T}%
\end{matrix}
\right) \left( 
\begin{matrix}
0 & \mathbf{X} \\ 
\mathbf{X}^{\ast} & 0%
\end{matrix}
\right) \left( 
\begin{matrix}
\mathbf{R} & 0 \\ 
0 & \mathbf{T}%
\end{matrix}
\right) , 
\end{equation*}
the spectrum of $\widehat{\mathbf{H}}$ is the same as the spectrum of 
\begin{equation}
\left( 
\begin{matrix}
\mathbf{R}^{2} & 0 \\ 
0 & \mathbf{T}^{2}%
\end{matrix}
\right) \left( 
\begin{matrix}
0 & \mathbf{X} \\ 
\mathbf{X}^{\ast} & 0%
\end{matrix}
\right) =\mathbf{Q}\widehat{\mathbf{X}}\text{, say}.   \label{QC}
\end{equation}
The $\textnormal{M}_{2n}\left( \mathbb{C}\right) $-valued Cauchy transform
of $\widehat{\mathbf{X}}$ is well known (see \cite{Far2008}). Thus, we just
need to find the $\textnormal{M}_{2n}\left( \mathbb{C}\right) $-valued
Cauchy transform of $\mathbf{Q}$ in order to be able to apply the
operator-valued subordination theory \cite{Belinschi2013a,Belinschi2013b};
see (\ref{CTOperProd}) in Theorem \ref{Thm:Subbordination} of the Appendix A.

\begin{remark}
a) The above mentioned operator valued subordination theory allows us to
compute, via iterative algorithms over matrices, the distribution of sums
and products of operator valued random variables free over some algebra
(Theorem \ref{Thm:Subbordination} in Appendix A). For a rigorous exposition
of the concept of freeness over an algebra, we refer the reader to \cite%
{Far2008,Speicher2012} and the references therein. Observe that this
relation is similar to the usual freeness in free probability.

b) The Cauchy transform of $\widehat{\mathbf{X}}$ is not given explicitly,
instead, it is given as a solution of a fixed point equation \cite{Far2008}.
In general, the $\textnormal{M}_{2n}\left( \mathbb{C}\right) $-valued Cauchy
transform $G_{\widehat{\mathbf{X}}}:\textnormal{M}_{2n}\left( \mathbb{C}%
\right) \rightarrow\textnormal{M}_{2n}\left( \mathbb{C}\right) $ has to be
computed for any matrix $B\in\textnormal{M}_{2n}\left( \mathbb{C}\right) $.
However, in some cases it is enough to compute it for diagonal matrices,
which simplifies the practical implementation (see Section 4).
\end{remark}

If the correlation matrices associated to the correlation operators $%
\{r_{k}\}$ are either constant or exhibit a distribution invariant under
random rotations, as we supposed for $r$ and $t$ in the previous section,
then these correlation operators will be free among themselves. In some
applications, these correlation operators come from constant matrices since
they model the antenna array architecture which in principle is fixed.
Suppose that this is the case, and that also the $\{t_{k}\}$ are free among
themselves. In some cases this hypothesis will be unnecessary (see Section %
\ref{Section:DiagonalCase}). Observe that 
\begin{equation*}
\mathbf{Q}=\sum_{k=1}^{n}r_{k}^{2}E_{k,k}+%
\sum_{k=1}^{n}t_{k}^{2}E_{n+k,n+k}. 
\end{equation*}
By the assumed freeness relations between the random variables $%
\{r_{k},t_{k}\}_{k}$, we have that the coefficients of the operator-valued
matrices in the previous sums are free, and thus the operator-valued
matrices $\{r_{k}^{2}E_{k,k}:1\leq k\leq n\}\cup\{t_{k}^{2}E_{n+k,n+k}:1\leq
k\leq n\}$ are free over $\textnormal{M}_{2n}\left( \mathbb{C}\right) $. So
we just have to compute the $\textnormal{M}_{2n}\left( \mathbb{C}\right) $%
-valued Cauchy transform of each operator-valued matrix in the above sum,
and then apply the results from the free additive subordination theory ((\ref%
{CTOperSum}) of Theorem \ref{Thm:Subbordination} in Appendix A).

\begin{theorem}
\label{Thm:CauchyTrk2} Let $r$ be a noncommutative random variable, $n\geq1$
a fixed integer and $k\in\{1,\ldots,n\}$. For $B\in\textnormal{M}_{2n}\left( 
\mathbb{C}\right) $, 
\begin{align*}
G_{rE_{k,k}}(B)=B^{-1}+[B^{-1}]_{k,k}^{-2}\left(
G_{r}([B^{-1}]_{k,k}^{-1})-[B^{-1}]_{k,k}\right) B^{-1}E_{k,k}B^{-1}.
\end{align*}
\end{theorem}

\begin{proof}
See Appendix C.
\end{proof}

With the previous theorem, we can compute the $\textnormal{M}_{2n}\left( 
\mathbb{C}\right) $-valued Cauchy transforms of $\{r_{k}^{2}E_{k,k},$ $%
t_{k}^{2}E_{n+k,n+k}\}$. With these transforms, we have all the elements to
compute the $\textnormal{M}_{2n}\left( \mathbb{C}\right) $-valued Cauchy
transform of $\ \widehat{\mathbf{H}}$, and in consequence the scalar Cauchy
transform of the spectrum $F$ of $\mathbf{HH}^{\ast}$ is obtained from (\ref%
{CTOpSca}): 
\begin{equation*}
G_{F}(\zeta)=\text{tr}_{n_{R}}(G_{\mathbf{HH}^{\ast}}(\zeta\mathrm{I}_{2n}))%
\mathbf{,\quad}\zeta\in\mathbb{C}\mathbf{.}
\end{equation*}
Using the Stieltjes inversion formula, one then obtains $F$ and this gives
the asymptotic isotropic mutual information (\ref{AsymptoticCapacity}).

\section{Channels with Symmetric-Like Behavior}

\label{Section:DiagonalCase}

From Section \ref{Section:Model}, we have that 
\begin{equation*}
\mathbf{X}=\left( 
\begin{matrix}
x_{1,1} & \cdots & x_{1,n} \\ 
\vdots & \ddots & \vdots \\ 
x_{n,1} & \cdots & x_{n,n}%
\end{matrix}
\right) 
\end{equation*}
is an operator-valued matrix composed of circular random variables with
correlation 
\begin{equation*}
\Sigma^{2}=\mathbb{E}\left( \Vec{H}\Vec{H}^{\ast}\right) . 
\end{equation*}
Observe that in this case, 
\begin{equation*}
\Vec{\mathbf{X}}=\Sigma\left( 
\begin{matrix}
c_{1,1} \\ 
c_{2,1} \\ 
\vdots \\ 
c_{n,n}%
\end{matrix}
\right) 
\end{equation*}
where the $c_{k,l}$ ($1\leq k,l\leq n$) are free circular random variables.
Thus there exist complex matrices $M_{k,l}$ for $1\leq k,l\leq n$ such that 
\begin{equation}
\mathbf{X}=\sum_{k,l=1}^{n}c_{k,l}M_{k,l},   \label{Eq:SumFreeOverM}
\end{equation}
i.e., $\mathbf{X}$ can be written as the sum of free circular random
variables multiplied by some complex matrices. In this way, the summands in (%
\ref{Eq:SumFreeOverM}) are free over $\textnormal{M}_{n}\left( \mathbb{C}%
\right) $. Observe that the previous procedure is exactly the same as
writing a matrix of complex Gaussian random variables as a sum of
independent complex Gaussian random variables multiplied by some complex
matrices.

For $1\leq k,l\leq n$, define 
\begin{equation*}
\widehat{\mathbf{X}}_{k,l} = \left( 
\begin{matrix}
\mathbf{0} & c_{k,l} M_{k.l} \\ 
c_{k,l}^{*}M_{k,l}^{*} & \mathbf{0}%
\end{matrix}
\right) , 
\end{equation*}
so $\widehat{\mathbf{X}}=\sum_{k,l=1}^{n} \widehat{\mathbf{X}}_{k,l}$.
Recall that the operator-valued matrices $\{\widehat{\mathbf{X}}_{k,l}\}$
are free over $\textnormal{M}_{2n}\left( \mathbb{C}\right) $. As an
alternative to the technique given in \cite{Far2008} to compute the
operator-valued Cauchy transform of $\widehat{\mathbf{X}}$, we can use the
subordination theory by computing the individual operator-valued Cauchy
transforms $G_{\widehat{\mathbf{X}}_{k,l}}$ for all $1\leq k,l\leq n$ and
then using equation (\ref{CTOperSum}). This technique is particularly neat
in the following setup.

Suppose that for all $1\leq k,l\leq n$ the operator-valued Cauchy transforms 
$G_{\widehat{\mathbf{X}}_{k,l}}$ send diagonal matrices to diagonal
matrices. From this assumption and Equations (\ref{CTOperSum}) and (\ref%
{CTOperProd}), it follows that this property is also shared by $G_{\widehat{%
\mathbf{X}}}$. Moreover, the following theorem shows that this is also true
for the operator-valued Cauchy transform of $\mathbf{Q}$.

\begin{theorem}
\label{Thm:CauchyTQ} Let $D=\textnormal{diag}\left(
d_{1},\ldots,d_{2n}\right) $ be a diagonal matrix in $\textnormal{M}%
_{2n}\left( \mathbb{C}\right) $. Then 
\begin{equation}  \label{Eq:CauchyTQ}
G_{\mathbf{Q}}(D) = \textnormal{diag}\left(
G_{r_{1}}(d_{1}),\ldots,G_{t_{n}}(d_{2n})\right) .
\end{equation}
\end{theorem}

\begin{proof}
See Appendix C.
\end{proof}

Since this diagonal invariance property is also satisfied by $\mathbf{Q}$,
again from Equations (\ref{CTOperSum}) and (\ref{CTOperProd}), we conclude
that $\widehat{\mathbf{H}}$ satisfies this property. Therefore, all the
computations involved in this case are within the framework of diagonal
matrices.

Also, in this diagonal case, any assumption of freeness between the
noncommutative random variables in $\mathbf{R}$ and $\mathbf{T}$ is
unnecessary since they do not interact when evaluating the Cauchy transform
of $\mathbf{Q}$ in diagonal matrices. Intuitively, the structure of $\mathbf{%
X}$ behaves well enough to destroy the effect that any possible dependency
between the correlation operators may have in the spectrum of $\mathbf{H}$.

It is easy to prove that the condition that $G_{\widehat{\mathbf{X}}_{k,l}}$
sends diagonal matrices to diagonal matrices is equivalent to requiring that 
\begin{equation*}
\left( 
\begin{matrix}
0 & M_{k,l} \\ 
M_{k,l}^{\ast} & 0%
\end{matrix}
\right) J\left( 
\begin{matrix}
0 & M_{k,l} \\ 
M_{k,l}^{\ast} & 0%
\end{matrix}
\right) 
\end{equation*}
is diagonal for any diagonal matrix $J\in\textnormal{M}_{2n}\left( \mathbb{C}%
\right) $. This last condition can be shown to be equivalent to requiring
that for all $1\leq k,l\leq n$, we have that $M_{k,l}=D_{k,l}P_{k,l}$ where $%
D_{,lk}$ is a diagonal matrix in $\textnormal{M}_{n}\left( \mathbb{C}\right) 
$ and $P_{k,l} $ is a permutation matrix.

\begin{remark}
If in a concrete application the correlation matrix $\mathbb{E}\left( \Vec {H%
}\Vec{H}^{\ast}\right) $ can be suitably decomposed, or approximated, in
such a way that this latter condition holds, then the method of this example
can be applied.
\end{remark}

\begin{theorem}
\label{Thm:CauchyTXkl} Let $n\geq1$. Suppose that $x$ is a circular random
variable, $D$ a diagonal matrix in $\textnormal{M}_{n}\left( \mathbb{C}%
\right) $, and $P$ a permutation matrix of the same size. Let $M:=DP$ and $%
\widehat{Mx}:=\left( 
\begin{smallmatrix}
0 & Mx \\ 
M^{\ast}x^{\ast} & 0%
\end{smallmatrix}
\right) $. Then, for $J=\left( 
\begin{smallmatrix}
J_{1} & 0 \\ 
0 & J_{2}%
\end{smallmatrix}
\right) $ with $J_{1}$ and $J_{2}$ diagonal matrices in $\textnormal{M}%
_{n}\left( \mathbb{C}\right) $, 
\begin{align}
G_{\widehat{Mx}}(J) & =\mathrm{diag}([J_{2}]_{\pi(1)}|D_{1}|^{-2}G_{xx^{\ast
}}([J_{1}]_{1}[J_{2}]_{\pi(1)}|D_{1}|^{-2}),\ldots  \notag
\label{Eq:CauchyTXkl} \\
&
\quad\quad\ldots,[J_{1}]_{\pi^{-1}(n)}|D_{\pi^{-1}(n)}|^{-2}G_{x^{%
\ast}x}([J_{1}]_{\pi^{-1}(n)}[J_{2}]_{n}|D_{\pi^{-1}(n)}|^{-2})).
\end{align}
\end{theorem}

\begin{proof}
See Appendix C.
\end{proof}

It is important to remark that $xx^{\ast}$ and $x^{\ast}x$ have a
Marchenko--Pastur distribution, of which the scalar Cauchy transform is
given by 
\begin{equation}
G_{xx^{\ast}}(\zeta)=\frac{\zeta-\sqrt{(\zeta-2)^{2}-4}}{2\zeta},\mathbf{%
\quad}\zeta\in\mathbb{C}.   \label{Eq:CauchyTMP}
\end{equation}

Taking $x=x_{k,l}$ and $M=M_{k,l}=D_{k,l}P_{k,l}$, we obtain the $%
\textnormal{M}_{2n}\left( \mathbb{C}\right) $-valued Cauchy transform of $%
\widehat{\mathbf{X}}_{k,l}$ explicitly. Given the scalar Cauchy transforms
of the variables $\{r_{k},t_{k}\}$, the corresponding operator-valued
transform of $\mathbf{Q}$ is also explicit, as given in Equation (\ref%
{Eq:CauchyTQ}). Nonetheless, the operator-valued Cauchy transform of $%
\widehat{\mathbf{X}}$ and $\mathbf{Q}\widehat{\mathbf{X}}$ are not given
explicitly, and need to be computed by means of Equations (\ref{CTOperSum})
and (\ref{CTOperProd}), respectively.

\subsection{Example}

Suppose that we have an operator-valued equivalent given by 
\begin{equation*}
\mathbf{H}=\left( 
\begin{matrix}
r_{1} & 0 \\ 
0 & r_{2}%
\end{matrix}
\right) \left( 
\begin{matrix}
x_{1} & x_{2} \\ 
x_{2} & x_{1}%
\end{matrix}
\right) \left( 
\begin{matrix}
t_{1} & 0 \\ 
0 & t_{2}%
\end{matrix}
\right) 
\end{equation*}
which corresponds to a channel with symmetric behavior. Let $\widehat{%
\mathbf{X}}_{1}$ and $\widehat{\mathbf{X}}_{2}$ be defined as follows 
\begin{equation*}
\widehat{\mathbf{X}}_{1}=x_{1}\left( 
\begin{matrix}
0 & 0 & x_{1} & 0 \\ 
0 & 0 & 0 & x_{1} \\ 
x_{1}^{*} & 0 & 0 & 0 \\ 
0 & x_{1}^{*} & 0 & 0%
\end{matrix}
\right) ;\ \ \ \widehat{\mathbf{X}}_{2}=x_{2}\left( 
\begin{matrix}
0 & 0 & 0 & x_{2} \\ 
0 & 0 & x_{2} & 0 \\ 
0 & x_{2}^{*} & 0 & 0 \\ 
x_{2}^{*} & 0 & 0 & 0%
\end{matrix}
\right) . 
\end{equation*}
In the notation of (\ref{QC}), $\widehat{\mathbf{X}}=\widehat{\mathbf{X}}%
_{1}+\widehat{\mathbf{X}}_{2}$. Moreover, using the same notation as above, $%
M_{1}=P_{1}=D_{1}=\mathrm{I}_{2}$, $M_{2}=P_{2}=\left( 
\begin{smallmatrix}
0 & 1 \\ 
1 & 0%
\end{smallmatrix}
\right) $ and $D_{2}=\mathrm{I}_{2}$. By Equation (\ref{Eq:CauchyTXkl}), the 
$\textnormal{M}_{4}\left( \mathbb{C}\right) $-valued Cauchy transforms of $%
\widehat{\mathbf{X}}_{1}$ and $\widehat{\mathbf{X}}_{1}$ are given, for $D=%
\textnormal{diag}\left( d_{1},d_{2},d_{3},d_{4}\right) $, by\footnote{%
Here we take the generic notation $xx^{*}$ to denote that $G_{xx^{*}}$ is
the scalar Cauchy transform in Equation (\ref{Eq:CauchyTMP}).} 
\begin{align*}
G_{\widehat{\mathbf{X}}_{1}}(D) & =\textnormal{diag}\left(
d_{3}G_{xx^{*}}(d_{1}d_{3}),d_{4}G_{xx^{*}}(d_{2}d_{4}),d_{1}G_{xx^{*}}(d_{1}d_{3}),d_{2}G_{xx^{*}}(d_{2}d_{4})\right)
\\
G_{\widehat{\mathbf{X}}_{2}}(D) & =\textnormal{diag}\left(
d_{4}G_{xx^{*}}(d_{1}d_{4}),d_{3}G_{xx^{*}}(d_{2}d_{3}),d_{2}G_{xx^{*}}(d_{3}d_{2}),d_{1}G_{xx^{*}}(d_{4}d_{1})\right)
\end{align*}
respectively.

Figure \ref{fig:Histogram} shows the asymptotic spectrum of $\mathbf{HH}%
^{\ast}$ against the corresponding matrix of size $1000\times1000$ when the
correlations $\{r_{k}^{2},t_{k}^{2}\}$ are assumed to obey the uniform
distribution on $[0,1]$. The figure shows good agreement.

\begin{figure}[th]
\centering
\includegraphics[width=0.5\textwidth]{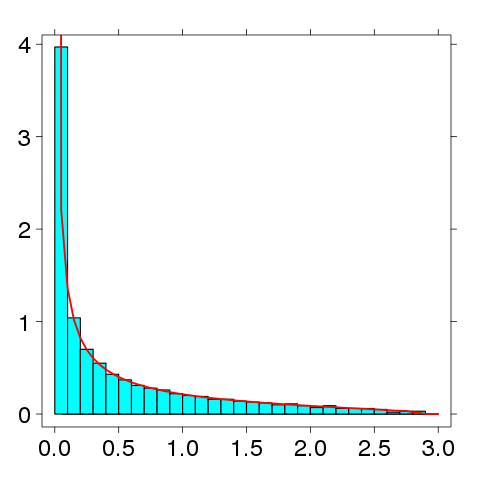} 
\caption{Histograms of the eigenvalues against the computed density.}
\label{fig:Histogram}
\end{figure}

\begin{remark}
Other symmetric-like channels can also be solved using the above approach,
for example 
\begin{equation*}
\mathbf{X}=\left( 
\begin{matrix}
x_{1} & x_{2} & x_{3} \\ 
x_{3} & x_{1} & x_{2} \\ 
x_{2} & x_{3} & x_{1}%
\end{matrix}
\right) ,\quad\quad\quad\mathbf{X}=\left( 
\begin{matrix}
x_{1} & x_{4} & x_{5} \\ 
x_{4} & x_{2} & x_{6} \\ 
x_{5} & x_{6} & x_{3}%
\end{matrix}
\right) . 
\end{equation*}
Observe that neither the matrix computed in this example nor the above
matrices have a separable correlation matrix.
\end{remark}

\section{Comparison With Other Models}

\label{Section:NumericalComparison}

In order to compare the operator-valued Kronecker model with some of the
classical models, in this section we compute the isotropic mutual
information of a $2\times2$ multiantenna system with Kronecker correlation
given by 
\begin{equation*}
K:=\frac{1}{8}\left( 
\begin{matrix}
1 & 0 \\ 
0 & 3%
\end{matrix}
\right) \otimes\left( 
\begin{matrix}
3 & 0 \\ 
0 & 5%
\end{matrix}
\right) , 
\end{equation*}
the asymptotic isotropic mutual information predicted by the usual Kronecker
correlation model, and the corresponding quantity based on the
operator-valued model. For such a channel, one possibility for implementing
the classical Kronecker correlation model is to take three noncommutative
random variables $r$, $x$ and $t$ such that $x$ is circular and the
distributions of $r^{2}$ and $t^{2} $ are given by 
\begin{align*}
\mu_{r^{2}} & =\frac{1}{2}\delta_{\frac{1}{2}}+\frac{1}{2}\delta_{\frac{3}{2}%
}, \\
\mu_{t^{2}} & =\frac{1}{2}\delta_{\frac{3}{4}}+\frac{1}{2}\delta_{\frac{5}{4}%
}.
\end{align*}
From this it is clear that we may compute the asymptotic isotropic mutual
information of the classical Kronecker model within the framework of the
operator-valued Kronecker model. In particular, the classical Kronecker
correlation model corresponds to the $n=1$ operator-valued Kronecker model.
This shows that the operator-valued Kronecker model is a generalization of
the usual Kronecker model also from this operational point of view.

The operator-valued Kronecker model uses $\Sigma^{2}=K$, but we have to use
a model for the correlation produced by the asymptotic antenna patterns.
Here we use two types of antenna pattern correlations. In one case we assume
that the distribution of the correlation operators $\{r_{k},t_{k}\}$ take 1
with probability one, i.e., there is no correlation due to the antenna
patterns; in the second case we assume that their distribution is given by 
\begin{equation}
\mu=\frac{18}{38}\delta_{1}+\frac{12}{38}\delta_{\frac{1}{2}}+\frac{8}{38}%
\delta_{\frac{1}{4}}.
\end{equation}
This distribution is motivated by an exponential decay law. In both cases we
set $\gamma=1$.

Figure \ref{fig:SNR} shows the mutual information of each model. The mutual
information of the $2\times2$ system was computed using a Monte Carlo
simulation. From this figure, we observe that the highest mutual information
is produced by the $2\times2 $ system. This is caused by the tail of the
eigenvalue distribution of the $2\times2$ random matrix involved. It is also
important to notice that the operator-valued model predicts more mutual
information than the usual Kronecker model when we assume no antenna pattern
correlations. However, in the presence of antenna pattern correlations, the
mutual information predicted by the operator-valued Kronecker model goes
below the one predicted by the classical Kronecker model. In particular,
this shows that the impact of the antenna design may be more significant
than the impact of the propagation environment itself.

\begin{figure}[th]
\centering
\includegraphics[width=0.6\textwidth]{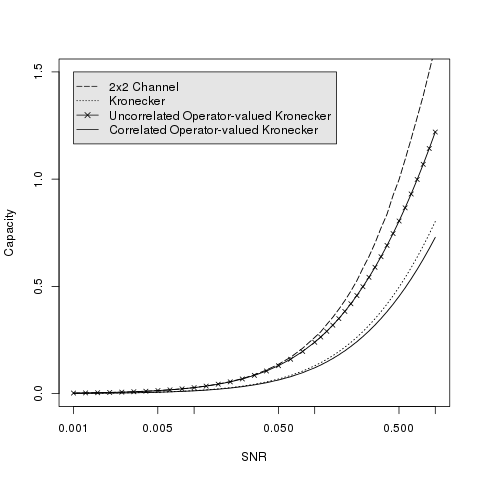} 
\caption{Isotropic mutual information predicted by the different models with
respect to $P$.}
\label{fig:SNR}
\end{figure}

\begin{remark}
Observe that in this example the correlation satisfies the hypothesis of the
previous section. In particular, 
\begin{align*}
\scriptstyle \left( 
\begin{matrix}
\sqrt{\frac{1}{2}} & 0 \\ 
0 & \sqrt{\frac{3}{2}}%
\end{matrix}
\right) \left( 
\begin{matrix}
x_{1} & x_{2} \\ 
x_{3} & x_{4}%
\end{matrix}
\right) \left( 
\begin{matrix}
\sqrt{\frac{3}{4}} & 0 \\ 
0 & \sqrt{\frac{5}{4}}%
\end{matrix}
\right) = \left( 
\begin{matrix}
\sqrt{\frac{3}{8}}x_1 & 0 \\ 
0 & 0%
\end{matrix}
\right) + \left( 
\begin{matrix}
0 & \sqrt{\frac{5}{8}}x_2 \\ 
0 & 0%
\end{matrix}
\right)+ \left( 
\begin{matrix}
0 & 0 \\ 
\sqrt{\frac{9}{8}}x_3 & 0%
\end{matrix}
\right) + \left( 
\begin{matrix}
0 & 0 \\ 
0 & \sqrt{\frac{15}{8}}x_4%
\end{matrix}
\right) .
\end{align*}
This shows that the operator-valued Kronecker model may be used for some
specific separable correlation channels.
\end{remark}

\appendix

\section{Prerequisites}

\subsection{Notation}

\begin{description}
\item $\mathbb{N}$: the set of natural numbers;

\item $\textnormal{M}_{n\times m}\left( \mathcal{C}\right) $: the set of all 
$n\times m$ matrices with entries from the algebra $\mathcal{C}$;

\item $A_{i,j}$ or $[A]_{i,j}$: the $i,j$th entry of the matrix $A$;

\item $A^{\top}$ the transpose of the matrix $A$, and $A^{\ast}$, its
conjugate transpose;

\item $E_{i,j}$: the $i,j$-unit matrix in $\textnormal{M}_{n\times m}\left( 
\mathbb{C}\right) $;

\item $\mathrm{I}_{n}$: the identity matrix in $\textnormal{M}_{n}\left( 
\mathbb{C}\right) $;

\item $\mathbb{E}$: $\ $expected valued with respect to a classical
probability space $(\Omega,\mathcal{F},\mathbb{P})$;
\end{description}

\subsection{Operator-Valued Free Probability Background}

In what follows, $\mathcal{C}$ will denote a noncommutative unital $C^{\ast} 
$-algebra with unit $\mathbf{1}_{\mathcal{C}}$, and $\varphi:\mathcal{C}%
\rightarrow\mathbb{C}$ is a unit-preserving positive linear functional,
i.e., $\varphi\left( \mathbf{1}_{\mathcal{C}}\right) =1$ and $\varphi\left(
aa^{\ast}\right) \geq0$ for any $a\in\mathcal{C}$. The pair $(\mathcal{C}%
,\varphi)$ is called a noncommutative probability space and the elements of $%
\mathcal{C}$ are called noncommutative random variables. Unless otherwise
stated, we use Greek letters to denote scalar numbers, lower case letters
for noncommutative random variables in $\mathcal{C}$, upper case letters for
matrices or random matrices in $\textnormal{M}_{n}\left( \mathbb{C}\right) $%
, and upper case bold letters for matrices in $\textnormal{M}_{n}\left( 
\mathcal{C}\right) $. The latter are called operator-valued matrices and $(%
\textnormal{M}_{n}\left( \mathcal{C}\right) ,\text{tr}_{n}\otimes\varphi)$
is a noncommutative probability space \cite{Speicher2012}.

Given a selfadjoint element $a\in\mathcal{C}$, its algebraic distribution is
the collection of its moments, i.e., $(\varphi\left( a^{k}\right) )_{k\geq 1}
$. Let $(\mathcal{A},\varphi)$ and $(\mathcal{A}_{n},\varphi_{n})$ for $%
n\geq1$ be noncommutative probability spaces. If $a\in\mathcal{A}$ and $%
a_{n}\in\mathcal{A}_{n}$ for $n\geq1$ are selfadjoint elements, we say that $%
(a_{n})_{n\geq1}$ converges in distribution to $a$ as $n\rightarrow\infty$
if the corresponding moments converge, i.e., 
\begin{equation*}
\lim_{n\rightarrow\infty}\varphi_{n}(a_{n}^{m})=\varphi\left( a^{m}\right) 
\end{equation*}
for all $m\in\mathbb{N}$. If there is a probability measure $\mu$ in $%
\mathbb{C}$ with compact support such that for all $m\in\mathbb{N}$ 
\begin{equation*}
\varphi\left( a^{m}\right) =\int_{\mathbb{C}}\zeta^{m}\,\mu(\mathrm{d}%
\zeta), 
\end{equation*}
we call $\mu$ the analytical distribution of $a$. A family $a_{1},\ldots
,a_{n}\in\mathcal{A}$ of noncommutative random variables is said to be 
\textit{free} if 
\begin{equation*}
\varphi\left( \lbrack p_{1}(a_{i_{1}})-\varphi\left( p_{1}(a_{i_{1}})\right)
]\cdots\lbrack p_{k}(a_{i_{k}})-\varphi\left( p_{k}(a_{i_{k}})\right)
]\right) =0 
\end{equation*}
for all $k\in\mathbb{N}$, polynomials $p_{1},\ldots,p_{k}$ and $%
i_{1},\ldots,i_{k}\in\{1,\ldots,n\}$ such that $i_{l}\neq i_{l+1}$ for $%
1\leq l\leq k-1$. Let $A_{n}$ and $B_{n}$ be random matrices in $\textnormal{%
M}_{n}\left( \mathbb{C}\right) $ for every $n\geq1$. If there exists $a,b\in%
\mathcal{C}$ such that $a$ and $b$ are free and $(A_{n},B_{n})$ converge in
distribution to $(a,b)$, i.e., 
\begin{equation*}
\lim_{n\rightarrow\infty}\frac{1}{n}\text{tr}\left(
A_{n}^{l_{1}}B_{n}^{m_{1}}\cdots A_{n}^{l_{k}}B_{n}^{m_{k}}\right)
=\varphi\left( a^{l_{1}}b^{m_{1}}\cdots a^{l_{k}}b^{m_{k}}\right) 
\end{equation*}
for all $k,l_{1},\ldots,l_{k},m_{1},\ldots,m_{k}\in\mathbb{N}$, we say that $%
A_{n}$ and $B_{n}$ are asymptotically free.

Given a probability measure $\mu$ in $\mathbb{R}$, its (scalar) Cauchy
transform $G_{\mu}:\mathbb{C}^{+}\rightarrow\mathbb{C}^{-}$ is defined as 
\begin{equation*}
G_{\mu}(\zeta):=\int_{\mathbb{R}}\frac{\mu(\text{d}\xi)}{\zeta-\xi}. 
\end{equation*}
The Stieltjes inversion formula states that if $\mu$ has density $f:\mathbb{R%
}\rightarrow\mathbb{R}$ then 
\begin{equation*}
f(\xi)=-\frac{1}{\pi}\lim_{%
\begin{smallmatrix}
\zeta\in\mathbb{R} \\ 
\zeta\rightarrow0+%
\end{smallmatrix}
}\Im(G_{\mu}(\xi+i\zeta)) 
\end{equation*}
for all $\xi\in\mathbb{R}$, where $\Im$ denotes the imaginary part and $\Re$
the real part.

Let $\mathcal{H}^{+}(\textnormal{M}_{n}\left( \mathbb{C}\right) )\subset 
\textnormal{M}_{n}\left( \mathbb{C}\right) $ denote the set of matrices $B$
such that $\Im(B):=\dfrac{B-B^{\ast}}{2i}$ is positive definite, and define $%
\mathcal{H}^{-}(\textnormal{M}_{n}\left( \mathbb{C}\right) ):=-\mathcal{H}%
^{+}(\textnormal{M}_{n}\left( \mathbb{C}\right) )$. For an operator-valued
matrix $\mathbf{X}\in\textnormal{M}_{n}\left( \mathcal{C}\right) $ we define
its $\textnormal{M}_{n}\left( \mathbb{C}\right) $-valued Cauchy transform $%
G_{\mathbf{X}}:\mathcal{H}^{+}(\textnormal{M}_{n}\left( \mathbb{C}\right)
)\rightarrow\mathcal{H}^{-}(\textnormal{M}_{n}\left( \mathbb{C}\right) )$ by 
\begin{align}
G_{\mathbf{X}}(B) & =E\left( (B-\mathbf{X})^{-1}\right)  \label{CTOper} \\
& =\sum_{n\geq0}B^{-1}E\left( (\mathbf{X}B^{-1})^{n}\right) ,  \notag
\end{align}
where the last power series converges in a neighborhood of infinity. The
scalar Cauchy transform of $\mathbf{X}$ is given by 
\begin{equation}
G(\zeta)=\text{tr}_{n} (G_{\mathbf{X}}(\zeta\mathrm{I}_{n}))\mathbf{,\quad }%
\zeta\in\mathbb{C}\mathbf{.}   \label{CTOpSca}
\end{equation}

The freeness relation over $\textnormal{M}_{n}\left( \mathbb{C}\right) $ is
defined similarly to the usual freeness, but taking $E$ instead of $\varphi$
and non-commutative polynomials over $\textnormal{M}_{n}\left( \mathbb{C}%
\right) $ instead of complex polynomials. The main tools that we use from
the subordination theory are the following formulas to compute the $%
\textnormal{M}_{n}\left( \mathbb{C}\right) $-valued Cauchy transforms of
sums and products of free elements in $\textnormal{M}_{n}\left( \mathcal{C}%
\right) $; see \cite{Belinschi2013a,Belinschi2013b}.

If $\mathbf{X}=\mathbf{X}^{\ast}$ is an operator-valued matrix in $%
\textnormal{M}_{n}\left( \mathcal{C}\right) $, we define the $r_{\mathbf{X}}$
and $h_{\mathbf{X}}$ transforms, for $B\in\mathcal{H}^{+}(\textnormal{M}%
_{n}\left( \mathbb{C}\right) )$, by 
\begin{align*}
r_{\mathbf{X}}(B) & =G_{\mathbf{X}}(B)^{-1}-B, \\
h_{\mathbf{X}}(B) & =B^{-1}-G_{\mathbf{X}}(B^{-1})^{-1}.
\end{align*}

\begin{theorem}
\label{Thm:Subbordination} Let $\mathbf{X},\mathbf{Y}\in\textnormal{M}%
_{n}\left( \mathcal{C}\right) $ be selfadjoint elements free over $%
\textnormal{M}_{n}\left( \mathcal{C}\right) $.

i) For all $B\in\mathcal{H}^{+}(\textnormal{M}_{n}\left( \mathbb{C}\right) ) 
$, we have that 
\begin{equation}
G_{\mathbf{X}+\mathbf{Y}}{B}=G_{\mathbf{X}}(\omega_{1}(B)), 
\label{CTOperSum}
\end{equation}
where $\omega_{1}(B)=\lim_{n\rightarrow\infty}f_{B}^{n}(W)$ for any $W\in%
\mathcal{H}^{+}(\textnormal{M}_{n}\left( \mathcal{C}\right) )$ and 
\begin{equation*}
f_{b}(W)=r_{\mathbf{Y}}(r_{\mathbf{X}}(W)+B)+B. 
\end{equation*}

ii) In addition, if $\mathbf{X}$ is positive definite, $E\left( \mathbf{X}%
\right) $ and $E\left( \mathbf{Y}\right) $ invertible, and we define for all 
$B\in\mathcal{H}^{+}(\textnormal{M}_{n}\left( \mathbb{C}\right) )$ with $%
\Im(B\mathbf{X})>0$ the function $g_{B}(W)=Bh_{\mathbf{X}}(h_{\mathbf{Y}%
}(W)B)$ for all $W\in\mathcal{H}^{+}(\textnormal{M}_{n}\left( \mathbb{C}%
\right) )$, then there exists a function $\omega_{2}$ such that 
\begin{equation*}
\omega_{2}(B)=\lim_{n\rightarrow\infty}g_{B}^{n}(W) 
\end{equation*}
for all $W\in\mathcal{H}^{+}(\textnormal{M}_{n}\left( \mathbb{C}\right) )$,
and 
\begin{align}
G_{\mathbf{XY}}(z\mathrm{I}_{n}) & =(z\mathrm{I}_{n}-h_{\mathbf{XY}}(z^{-1}%
\mathrm{I}_{n}))^{-1},  \label{CTOperProd} \\
zh_{\mathbf{XY}}(z\mathrm{I}_{n}) & =\omega_{2}(z\mathrm{I}_{n})h_{\mathbf{Y}%
}(\omega_{2}(z\mathrm{I}_{n}))).  \notag
\end{align}
\end{theorem}

The functions above are defined in $\mathcal{H}^{+}(\textnormal{M}_{n}\left( 
\mathbb{C}\right) )$. Whenever we evaluate any of these functions in $B\in%
\mathcal{H}^{-}(\textnormal{M}_{n}\left( \mathbb{C}\right) )$ we have to do
so by means of the relation $f(B)=f(B^{\ast})^{\ast}$.

\section{Proof of Theorem \protect\ref{Thm:ExtremalCases} and Further
Analysis}

\subsection{Case $\protect\gamma\to\infty$}

It is a well known result \cite{Tao2012} that the eigenvalues are continuous
functions of the entries of a selfadjoint matrix. If the entries of a matrix 
$M$ lie in the unit circle, then its Frobenius norm is bounded and so its
operator norm. In particular, $g(M):=(\lambda_1(MM^*),\ldots,\lambda_N(MM^*))
$ is a bounded and continuous function of the entries of $M$. Therefore, if
we prove that the entries of $A$ converge in distribution to the entries of $%
U$, i.e. $(A_{i,j})_{i,j=1}^N \overset{\textnormal{d}}{\to}
(U_{i,j})_{i,j=1}^N$, then $g(A)\overset{\textnormal{d}}{\to} g(U)$ as
required.

The entries of $A$ and $U$ lie in the unit circle, so we are dealing with
compact support distributions. Thus, it is enough to show the convergence of
the joint moments of the entries of $A$ to those of $U$ to ensure the
multivariate convergence in distribution, and so the claimed convergence in
the first part of Theorem \ref{Thm:ExtremalCases}.

Let $N\in\mathbb{N}$ be fixed, for $(n_{k,l})_{k,l=1}^N\subset\mathbb{Z}$ 
\begin{align*}
\mathbb{E}\left(\prod_{k,l=1}^N A_{k,l}^{n_{k,l}}\right) &= \mathbb{E}%
\left(\prod_{k,l=1}^N \exp\left(\mathrm{i}\gamma
n_{k,l}\theta_{k,l}\right)\right) \\
&= \mathbb{E}\left(\prod_{k,l=1}^N \exp\left(\mathrm{i}\gamma
n_{k,l}\sum_{i,j=1}^N R_{k,i}X_{i,j}T_{j,l}\right)\right) \\
&= \mathbb{E}\left(\exp\left(\sum_{i,j=1}^N \mathrm{i} \gamma
\left(\sum_{k,l=1}^N n_{k,l} R_{k,i}T_{j,l}\right) X_{i,j} \right)\right) \\
&= \mathbb{E}\left(\prod_{i,j=1}^N \exp\left( \mathrm{i} \gamma
\left(\sum_{k,l=1}^N n_{k,l} R_{k,i}T_{j,l}\right) X_{i,j} \right)\right) \\
&= \prod_{i,j=1}^N \exp\left(-\frac{\gamma^2}{2} \left(\sum_{k,l=1}^N
n_{k,l} R_{k,i}T_{j,l}\right)^2 \right).
\end{align*}
Since $R$ and $T$ are full rank, a linear algebra argument shows that the
previous exponents are all zero if and only if $(n_{k,l})_{k,l=1}^N$ are all
zero. Therefore, the joint moments of the entries of $A$ vanish as $%
\gamma\rightarrow\infty$ except when $n_{k,l}=0$ for all $k$ and $l$. It is
easy to show that these limiting moments are indeed the joint moments of the
entries of $U$. This conclude the proof of the first part.

\subsection{Case $\protect\gamma\to0$}

The following lemma and two theorems are from Appendix A in \cite{Bai2010}

\begin{lemma}
\label{Lem:NormHadamardProd} Let $A_1,\ldots,A_l\in\textnormal{M}_{m\times
n}\left(\mathbb{C}\right)$. Then 
\begin{equation*}
\left|\left|A_1\circ A_2\circ\cdots\circ
A_l\right|\right|_{}\leq\left|\left|A_1\right|\right|_{}\left|\left|A_2%
\right|\right|_{}\cdots\left|\left|A_l\right|\right|_{},
\end{equation*}
where $A\circ B$ denotes the pointwise or Hadamard product of $A$ and $B$.
\end{lemma}

\begin{theorem}
\label{Thm:PerturbationInq} Let $A,B\in\textnormal{M}_{m\times n}\left(%
\mathbb{C}\right)$. Then 
\begin{equation*}
\sum_{k=1}^p |\sigma_k(A)-\sigma_k(B)|^2 \leq \text{tr}\left((A-B)(A-B)^*%
\right)
\end{equation*}
where $p=\min(m,n)$ and $\sigma_1(\cdot)\geq\cdots\geq\sigma_p(\cdot)$ are
the singular values of $\cdot$.
\end{theorem}

\begin{theorem}
\label{Thm:InqEigDistRank} Let $A$ and $B$ be two $m\times n$ complex
matrices. Then, for any Hermitian complex matrices $X\in\textnormal{M}%
_{m}\left(\mathbb{C}\right)$ and $Y\in\textnormal{M}_{n}\left(\mathbb{C}%
\right)$ we have that 
\begin{equation*}
||F^{X+AYA^*}-F^{X+BYB^*}|| \leq \frac{1}{m} \textnormal{rank}%
\left(A-B\right).
\end{equation*}
\end{theorem}

In this rest of this subsection, $F^A$ will denote the empirical
distribution of the \textit{singular values} $\sigma_1(A)\geq\cdots\geq%
\sigma_n(A)$ of $A\in\textnormal{M}_{n\times n}\left(\mathbb{C}\right)$.
Since the classical convergence theorems in random matrices hold almost
surely, it is enough to deal with the case of non-random matrices.

\begin{lemma}
\label{Lem:InqSingularValuesSqrtTrace} Let $A,B\in\textnormal{M}_{N}\left(%
\mathbb{C}\right)$. Then 
\begin{equation*}
\sum_{k=1}^N |\sigma_k(A)-\sigma_k(B)| \leq \sqrt{N\text{tr}%
\left((A-B)(A-B)^*\right)}.
\end{equation*}
\end{lemma}

\begin{proof}
An straightforward application of Theorem \ref{Thm:PerturbationInq} and the generalized means.
\end{proof}

\begin{definition}
We define the entrywise exponential function $\exp_\circ:\textnormal{M}%
_{m\times n}\left(\mathbb{C}\right)\to\textnormal{M}_{m\times n}\left(%
\mathbb{C}\right)$ by 
\begin{equation*}
\exp_\circ(A) = (\exp(A_{i,j}))_{i,j}
\end{equation*}
for all $A\in\textnormal{M}_{m\times n}\left(\mathbb{C}\right)$.
\end{definition}

\begin{proposition}
\label{Prop:EntrywiseExpDist} Let $A\in\textnormal{M}_{N}\left(\mathbb{C}%
\right)$ for $N\in\mathbb{N}$ and $1>\gamma>0$. Let $X=\exp_\circ(i\gamma A)$%
, then 
\begin{equation}
\frac{1}{N}\sum_{k=2}^N \left|\sigma_k\left(\frac{X}{\gamma}%
\right)-\sigma_k(A)\right| \leq \gamma \exp(\left|\left|A\right|\right|_{})+%
\frac{2||A||}{N}.
\end{equation}
\end{proposition}

\begin{proof}
Using the power series for the exponential function we obtain that
\begin{equation}
X={\bf 1}_N+i\gamma A+\sum_{n\geq2} \frac{(i\gamma A)^{\circ n}}{n!}
\end{equation}
where $T^{\circ n}=T\circ T\circ \cdots\circ T$. Define $Z={\bf 1}_N + i\gamma A$ and $Y=X-Z$. By Lemma \ref{Lem:NormHadamardProd} and the fact that $\gamma<1$,
\begin{align*}
\norm{}{Y} &= \gamma^2 \norm{}{\sum_{n\geq2} \gamma^{n-2} \frac{(iA)^{\circ n}}{n!}}\\
&\leq \gamma^2 \exp(\norm{}{A}).
\end{align*}
By Lemma \ref{Lem:InqSingularValuesSqrtTrace} we have that
\begin{align}
\nonumber \sum_{k=1}^N \left|\sigma_k(X)-\sigma_k(Z)\right| &\leq \sqrt{N\tr{YY^*}}\\
\nonumber &\leq \sqrt{N^2\norm{}{Y}^2}\\
\label{Eq:EstimateXZ} &\leq \gamma^2 N\exp(\norm{}{A})
\end{align}
and in particular
\begin{equation}
\frac{1}{N} \sum_{k=2}^N \left|\sigma_k\left(\frac{X}{\gamma}\right) - \sigma_k\left(\frac{Z}{\gamma}\right)\right| \leq \gamma \exp(\norm{}{A}).
\end{equation}
Applying Theorem \ref{Thm:InqEigDistRank} to the matrices $Z$ and $\gamma A$ we obtain\footnote{Recall that the singular values of $i\gamma A$ and $\gamma A$ are equal, i.e. $\sigma_k(i\gamma A)=\sigma_k(\gamma A)$ for all $1\leq k\leq n$.}
\begin{equation*}
\norm{}{F^{ZZ^*}-F^{AA^*}} \leq \frac{1}{N} \rank{{\bf 1}_N} = \frac{1}{N},
\end{equation*}
which implies that
\begin{equation*}
\left| \sum_{k=1}^N \I{x\leq \sigma_k(Z)^2} - \sum_{k=1}^N \I{x\leq\sigma_k(\gamma A)^2} \right| \leq 1
\end{equation*}
for all $x\in\R$. This implies that for $2\leq k\leq N-1$
\begin{equation}
\label{Eq:Interleaving}
\sigma_{k+1}(\gamma A)\leq \sigma_k(Z) \leq \sigma_{k-1}(\gamma A),
\end{equation}
and equivalently
\begin{equation*}
\sigma_{k+1}(\gamma A)-\sigma_k(\gamma A)\leq \sigma_k(Z)-\sigma_k(\gamma A) \leq \sigma_{k-1}(\gamma A)-\sigma_k(\gamma A).
\end{equation*}
Therefore
\begin{align*}
|\sigma_k(Z)-\sigma_k(\gamma A)| &\leq \sigma_{k-1}(\gamma A)-\sigma_k(\gamma A)+\sigma_k(\gamma A)-\sigma_{k+1}(\gamma A)\\
&= \sigma_{k-1}(\gamma A)-\sigma_{k+1}(\gamma A),
\end{align*}
and consequently
\begin{align*}
\sum_{k=2}^N |\sigma_k(Z)-\sigma_k(\gamma A)| &\leq \sum_{k=2}^{N-1} \sigma_{k-1}(\gamma A)-\sigma_{k+1}(\gamma A)+|\sigma_N(Z)-\sigma_N(\gamma A)|\\
&\leq \sigma_1(\gamma A) + \sigma_2(\gamma A) - \sigma_{N-1}(\gamma A)-\sigma_N(\gamma A)+\sigma_N(Z)+\sigma_N(\gamma A).
\end{align*}
Using the same argument that in equation (\ref{Eq:Interleaving}) we have that $\sigma_N(Z)\leq\sigma_{N-1}(\gamma A)$ and thus
\begin{equation*}
\sum_{k=2}^N |\sigma_k(Z)-\sigma_k(\gamma A)| \leq 2\gamma ||A||
\end{equation*}
and in particular
\begin{equation*}
\frac{1}{N} \sum_{k=2}^N \left|\sigma_k\left(\frac{Z}{\gamma}\right)-\sigma_k(A)\right| \leq \frac{2||A||}{N}.
\end{equation*}
By the triangle inequality we conclude that
\begin{equation*}
\frac{1}{N}\sum_{k=2}^N \left|\sigma_k\left(\frac{X}{\gamma}\right)-\sigma_k(A)\right| \leq \gamma \exp(\norm{}{A})+\frac{2||A||}{N}
\end{equation*}
as claimed.
\end{proof}

Observe that the previous analysis exclude the biggest singular value of $%
X/\sigma$. In the following proposition we study the behavior of this
singular value.

\begin{proposition}
In the notation of the previous proposition, 
\begin{equation*}
\left|\frac{\sigma_1(X/\gamma)}{N/\gamma}-1\right| \leq \gamma (\gamma
\exp(\left|\left|A\right|\right|_{})+\left|\left|A\right|\right|_{}).
\end{equation*}
\end{proposition}

This shows that $\sigma_1(X/\gamma)$ is roughly $N/\gamma$, while the bulk
of $X/\gamma$ is essentially the same as $A$.

\begin{proof}
By inequality (\ref{Eq:EstimateXZ}) in the first part of the previous proof
\begin{equation}
\left|\sigma_1\left(\frac{X}{\gamma}\right)-\sigma_1\left(\frac{Z}{\gamma}\right)\right| \leq \gamma N \exp(\norm{}{A}).
\end{equation}
Using Lemma \ref{Lem:InqSingularValuesSqrtTrace} for $Z/\gamma$ and ${\bf 1}_N/\gamma$
\begin{align*}
\left|\sigma_1\left(\frac{Z}{\gamma}\right) - \sigma_1\left(\frac{{\bf 1}_N}{\gamma}\right)\right| &\leq \sqrt{N\tr{AA^*}}\\
&\leq N \norm{}{A}.
\end{align*}
A straightforward computation shows that $\sigma_1({\bf 1}_N/\gamma)=N/\gamma$, so by the triangle inequality
\begin{align*}
\left|\frac{\sigma_1(X/\gamma)}{N/\gamma}-1\right| \leq \gamma (\gamma \exp(\norm{}{A})+\norm{}{A}),
\end{align*}
as claimed.
\end{proof}

Finally, with the previous quantitative results we prove the following
qualitative result.

\begin{theorem}
Let $A_N\in\textnormal{M}_{N}\left(\mathbb{C}\right)$ such that $%
\left|\left|A_N\right|\right|_{}$ converge as $N\to\infty$ and $F^{A_N}
\Rightarrow F^A$. Define $X_N=\exp_\circ(i\gamma_N A_N)$. If $%
(\gamma_N)_{N\geq1}$ is a sequence of positive real numbers such that $%
\gamma_N\to0$ as $N\to\infty$, then $F^{X_N/\gamma_N} \Rightarrow F^A$ as $%
N\to\infty$.
\end{theorem}

\begin{proof}
Recall that $F^{X_N/\gamma_N}\Rightarrow F^A$ if and only if
\begin{equation*}
\lim_{N\to\infty} \int_\R f(x) \dif F^{X_N/\gamma_N}(x) = \int_\R f(x) \dif F^A(x)
\end{equation*}
for all $f$ bounded Lipschitz function. Let $f$ be any bounded Lipschitz function, by the previous propositions
\begin{align*}
\bigg| \int_\R f(x) \dif F^{X_N/\gamma_N}(x) &- \int_\R f(x) \dif F^{A_N}(x) \bigg|\\
&= \left| \frac{1}{N} \sum_{k=1}^N f\left(\sigma_k\left(\frac{X_N}{\gamma_N}\right)\right) - \frac{1}{N} \sum_{k=1}^N f(\sigma_k(A_N))\right|\\
&\leq \frac{1}{N} \sum_{k=1}^N \left|f\left(\sigma_k\left(\frac{X_N}{\gamma_N}\right)\right)-f(\sigma_k(A_N))\right|\\
&\leq \frac{K}{N} \sum_{k=2}^N \left|\sigma_k\left(\frac{X_N}{\gamma_N}\right)-\sigma_k(A_N)\right| +
\frac{\left|f\left(\sigma_1\left(\frac{X_N}{\gamma_N}\right)\right)\right|+|f(\sigma_1(A_N))|}{N},
\end{align*}
where $K$ is the Lipschitz constant of $f$. Since $f$ is bounded and $\norm{}{A_N}$ converge as $N\to\infty$, by Proposition \ref{Prop:EntrywiseExpDist} the previous expression converges to 0 as $N\to\infty$. Finally, since $F^{A_N} \Rightarrow F^A$ as $N\to\infty$ we have that
\begin{equation*}
\lim_{N\to\infty} \left| \int_\R f(x) \dif F^{A_N}(x) - \int_\R f(x) \dif F^A(x)\right| = 0
\end{equation*}
and by the triangle inequality the result follows.
\end{proof}

The second part of Theorem \ref{Thm:ExtremalCases} is an straightforward
application of the previous theorem.

\section{Computation of Some Cauchy Transforms}

\noindent\textit{Proof of Theorem \ref{Thm:CauchyTrk2}.} The identities $%
E_{k,k}BE_{k,k}=B_{k,k}E_{k,k}$ and $E_{k,k}^{2}=E_{k,k}$ lead to 
\begin{align*}
G_{r_{k}^{2}E_{k,k}}(B) & =\sum_{n\geq0}B^{-1}E\left(
(r_{k}^{2}E_{k,k}B^{-1})^{n}\right) \\
& =B^{-1}+B^{-1}\sum_{n\geq1}\varphi\left( r_{k}^{2n}\right)
[B^{-1}]_{k,k}^{n-1}E_{k,k}B^{-1} \\
& =B^{-1}+[B^{-1}]_{k,k}^{-2}\left( \sum_{n\geq0}\varphi\left(
r_{k}^{2n}\right) [B^{-1}]_{k,k}^{n+1}-[B^{-1}]_{k,k}\right)
B^{-1}E_{k,k}B^{-1} \\
& =B^{-1}+[B^{-1}]_{k,k}^{-2}\left(
G_{r_{k}^{2}}([B^{-1}]_{k,k}^{-1})-[B^{-1}]_{k,k}\right) B^{-1}E_{k,k}B^{-1}.
\end{align*}
Of course, the previous equations do not hold for every matrix $B\in 
\textnormal{M}_{2n}\left( \mathbb{C}\right) $, in particular, the power
series expansion is valid only in a neighborhood of infinity. However, the
previous computation can be carried out at the level of formal power series,
and then extended via analytical continuation to a suitable domain.

\noindent\textit{Proof of Theorem \ref{Thm:CauchyTQ}.} A straightforward
computation shows that 
\begin{align}
G_{\mathbf{Q}}(D) & =\sum_{k\geq0}D^{-1} E\left( (\mathbf{Q}%
D^{-1})^{k}\right)  \notag \\
& =\sum_{k\geq0}\textnormal{diag}\left( d_{1}^{-1},\ldots,d_{2n}^{-1}\right)
E\left( \textnormal{diag}\left(
d_{1}^{-k}r_{1}^{k},\ldots,d_{n}^{-k}r_{n}^{k},d_{n+1}^{-k}t_{1}^{k},%
\cdots,d_{2n}^{-k}t_{n}^{k}\right) \right)  \notag \\
& =\sum_{k\geq0}\textnormal{diag}\left( d_{1}^{-(k+1)}\varphi\left(
r_{1}^{k}\right) ,\ldots,d_{2n}^{-(k+1)}\varphi\left( t_{n}^{k}\right)
\right)  \notag \\
& =\textnormal{diag}\left( G_{r_{1}}(d_{1}),\ldots,G_{t_{n}}(d_{2n})\right) .
\notag
\end{align}

\noindent\textit{Proof of Theorem \ref{Thm:CauchyTXkl}.} Observe that 
\begin{equation*}
\widehat{Mx}J^{-1}=\left( 
\begin{matrix}
0 & DPx \\ 
(DP)^{\ast}x^{\ast} & 0%
\end{matrix}
\right) \left( 
\begin{matrix}
J_{1}^{-1} & 0 \\ 
0 & J_{2}^{-1}%
\end{matrix}
\right) =\left( 
\begin{matrix}
0 & DPJ_{2}^{-1}x \\ 
(DP)^{\ast}J_{1}^{-1}x^{\ast} & 0%
\end{matrix}
\right) . 
\end{equation*}
Thus, 
\begin{equation*}
\left( \widehat{Mx}J^{-1}\right) ^{2}=\left( 
\begin{matrix}
DPJ_{2}^{-1}(DP)^{\ast}J_{1}^{-1}xx^{\ast} & 0 \\ 
0 & (DP)^{\ast}J_{1}^{-1}DPJ_{2}^{-1}x^{\ast}x%
\end{matrix}
\right) . 
\end{equation*}
Since $P^{\top}D^{\prime}P$ and $PD^{\prime}P^{\top}$ are diagonal for any
diagonal matrix $D^{\prime}$, and diagonal matrices commute, we have for $%
n\geq1$ that 
\begin{equation*}
\left( \widehat{Mx}J^{-1}\right) ^{2n}=\left( 
\begin{matrix}
J_{1}^{-n}(DPJ_{2}^{-1}(DP)^{\ast})^{n}(xx^{\ast})^{n} & 0 \\ 
0 & ((DP)^{\ast}J_{1}^{-1}DP)^{n}J_{2}^{-n}(x^{\ast}x)^{n}%
\end{matrix}
\right) . 
\end{equation*}
Recalling that the odd moments of $x$ are zero, the previous equation
implies 
\begin{align*}
G_{\widehat{Mx}}(J) & =\sum_{n\geq0}J^{-1}E\left( \left( \widehat{Mx}%
J^{-1}\right) ^{n}\right) \\
& =\sum_{n\geq0}J^{-1}E\left( \left( \widehat{Mx}J^{-1}\right) ^{2n}\right)
\\
& =\sum_{n\geq0}\left( 
\begin{matrix}
J_{1}^{-(n+1)}(DPJ_{2}^{-1}(DP)^{\ast})^{n}\varphi\left(
(xx^{\ast})^{n}\right) & 0 \\ 
0 & ((DP)^{\ast}J_{1}^{-1}DP)^{n}J_{2}^{-(n+1)}\varphi\left(
(x^{\ast}x)^{n}\right)%
\end{matrix}
\right) .
\end{align*}
Finally, let $\pi$ be the permutation associated to $P$, then $[PD^{\prime
}]_{k}=[D^{\prime}]_{\pi(k)}$ for any diagonal matrix $D^{\prime}$ and any $%
1\leq k\leq n$. Therefore\footnote{%
For notational simplicity, let $D_{k}^{\prime}$ denote the $k,k$th entry of
the diagonal matrix $D^{\prime}$.} 
\begin{align}
G_{\widehat{Mx}}(J) & =\left( 
\begin{matrix}
(DPJ_{2}^{-1}(DP)^{\ast})^{-1} & 0 \\ 
0 & ((DP)^{\ast}J_{1}^{-1}DP)^{-1}%
\end{matrix}
\right) \times  \notag \\
& \quad\quad\sum_{n\geq0}\left( 
\begin{matrix}
J_{1}^{-(n+1)}(DPJ_{2}^{-1}(DP)^{\ast})^{n+1}\varphi\left(
(xx^{\ast})^{n}\right) & 0 \\ 
0 & ((DP)^{\ast}J_{1}^{-1}DP)^{n+1}J_{2}^{-(n+1)}\varphi\left( (x^{\ast
}x)^{n}\right)%
\end{matrix}
\right)  \notag \\
& =\left( 
\begin{matrix}
(DPJ_{2}^{-1}(DP)^{\ast})^{-1} & 0 \\ 
0 & ((DP)^{\ast}J_{1}^{-1}DP)^{-1}%
\end{matrix}
\right) \times  \notag \\
& \quad\quad\textnormal{diag}\left( G_{xx^{\ast}}([J_{1}]_{1}[J_{2}]_{\pi
(1)}|D_{1}|^{-2}),\ldots,G_{x^{\ast}x}([J_{1}]_{\pi^{-1}(n)}[J_{2}]_{n}|D_{%
\pi^{-1}(n)}|^{-2})\right)  \notag \\
& =\mathrm{diag}([J_{2}]_{\pi(1)}|D_{1}|^{-2}G_{xx^{%
\ast}}([J_{1}]_{1}[J_{2}]_{\pi(1)}|D_{1}|^{-2}),\ldots  \notag \\
&
\quad\quad\ldots,[J_{1}]_{\pi^{-1}(n)}|D_{\pi^{-1}(n)}|^{-2}G_{x^{%
\ast}x}([J_{1}]_{\pi^{-1}(n)}[J_{2}]_{n}|D_{\pi^{-1}(n)}|^{-2})).  \notag
\end{align}

\section*{Acknowledgment}

Mario Diaz was supported in part by the Centro de Investigaci\'{o}n en Matem%
\'{a}ticas A.C., M\'{e}xico and the government of Ontario, Canada.

\end{document}